\documentclass{siamltex}
\usepackage{amsmath}
\usepackage{amssymb}
\usepackage{graphicx}
\usepackage{subfigure}
\usepackage{algorithm}

\newcommand{\R}{\mathcal{R}}
\renewcommand{\P}{P}
\renewcommand{\t}{^{\top}}
\newcommand{\lf}{\text{\sc lfil}}
\newtheorem{Theorem}{Theorem}[section]
\newtheorem{Remark}[Theorem]{Remark}

\newcommand{\wP}{\widehat{P}_{k+1}}
\newcommand{\wQ}{\widetilde{Q}}
\newcommand{\ww}{\widehat{P}}

\newcommand{\wk}{\widehat{P}_{k}}
\newcommand{\se}{\sqrt{\|\fe_k\|}}
\newcommand{\IN}{\text{\sc Imax}}
\newcommand{\IL}{\text{\sc Imax$_{PCG}$}}
 \newcommand{\JJ}{J_{k+1}^{1/2}}
 \newcommand{\Jk}{J_{k}^{1/2}}
 \newcommand{\sJ}{J_{k}^{1/2}}
 \newcommand{\sn}{J_{0}^{1/2}}

\usepackage{geometry}
\usepackage{here}

\newcommand{\fc}{\mbox{\boldmath  $c$}}
\newcommand{\fe}{\mbox{\boldmath  $e$}}
\newcommand{\fg}{\mbox{\boldmath  $g$}}

\newcommand{\fr}{\mbox{\boldmath  $r$}}
\newcommand{\fs}{\mbox{\boldmath  $s$}}
\newcommand{\ft}{\mbox{\boldmath  $t$}}
\newcommand{\fu}{\mbox{\boldmath  $u$}}
\newcommand{\fv}{\mbox{\boldmath  $v$}}
\newcommand{\fw}{\mbox{\boldmath  $w$}}
\newcommand{\fx}{\mbox{\boldmath  $x$}}
\newcommand{\fy}{\mbox{\boldmath  $y$}}
\newcommand{\fz}{\mbox{\boldmath  $z$}}

\title{Efficiently preconditioned Inexact Newton methods for large symmetric eigenvalue problems}
\author{L. Bergamaschi\footnotemark[1] \footnotemark[2],
and A. Mart\'{\i}nez\footnotemark[4]\\[.4em] }

\begin{document}

\maketitle

\renewcommand{\thefootnote}{\fnsymbol{footnote}}

\footnotetext[1] {Corresponding author}
\footnotetext[2] {Department of Civil, Environmental and Architectural Engineering,
                  University of Padua, via Trieste 63, 35100 Padova, Italy, {\tt luca.bergamaschi@unipd.it}.  The work of this author has been partially supported by the Spanish grant
	  MTM2010-18674}
\footnotetext[4] {Department of Mathematics,
                  University of Padua, via Trieste 63, 35100 Padova, Italy, {\tt acalomar@dmsa.unipd.it}}

\renewcommand{\thefootnote}{\fnsymbol{footnote}}

\begin{abstract}
	In this paper we propose an efficiently preconditioned Newton method for the computation of the leftmost eigenpairs
	of large and sparse symmetric positive definite matrices.
	A sequence of preconditioners based on the BFGS update formula is proposed, for the Preconditioned Conjugate Gradient
	solution of the linearized Newton system  to solve $A \fu  = q(\fu) \fu$, $q(\fu)$ being the Rayleigh Quotient.  
	We give theoretical evidence that the sequence of preconditioned Jacobians remains close to the identity matrix if the initial
	preconditioned Jacobian is so. 
	Numerical results onto matrices arising from various realistic problems with size up to one million
	unknowns account for the efficiency of the proposed algorithm which reveals competitive with the Jacobi-Davidson method
	on all the test problems.
\end{abstract}

	\medskip


\medskip

\section{Introduction}
Consider a  symmetric positive definite (SPD) matrix $A$, which is also assumed to be
large and sparse.
We will denote as 
\[ \lambda_1 < \lambda_2 < \ldots < \lambda_p < \ldots < \lambda_n \]
the eigenvalues of $A$ and 
\[ \fv_1,  \fv_2, \ldots,  \fv_p, \ldots, \fv_n \]
the corresponding (normalized) eigenvectors.

The computation of the $p \ll n $ leftmost eigenpairs of $A$
is a common task in many scientific
applications.
Typical examples are offered by the vibrational analysis of
mechanical structures \cite{bat82}, the lightwave technology~\cite
{zob92}, electronic structure calculations \cite{saad-et-al-1995}, and the spectral superposition approach for the solution
of large sets of 1st order linear differential equations~\cite{gam93time}.
Computation of a few eigenpairs is also crucial in the approximation of the generalized inverse of
the graph Laplacian~\cite{BF_2012,BBF_2013}.

In this paper we propose to use an efficiently preconditioned Newton method for the nonlinear system of equations:
\begin{equation}
	\label{nonlinear}
	A \fu -  q(\fu) \fu = 0 \qquad  \text{where} \qquad
 q(\fu) = \dfrac{\fu\t A \fu}{\fu\t \fu} 
\end{equation}
is the Rayleigh Quotient.
The idea of employing the Newton method for this nonlinear system
is obviously not new:
among the others we mention Davidson (\cite{davidson}) who approximated the Jacobian of (\ref{nonlinear}) with $\text{diag}(A - q(\fu)I)$
and combined this system solution with a Rayleigh-Ritz procedure. 

The Newton  method in the unit sphere
\cite{simonciniRQI,freitag} or Newton-Grassman method, constructs a sequence of vectors $\{\fu_k\}$ by  solving the linear systems
\begin{equation}
	\label{gras}
(I - \fu_k \fu_k\t) (A - \theta_k I) (I - \fu_k \fu_k\t) \fs  = - (A \fu_k - \theta_k \fu_k),  \qquad
				      \theta_k = \dfrac{\fu_k\t A \fu_k}{\fu_k\t \fu_k}  
\end{equation}
ensuring that the correction $\fs$ be orthogonal to $\fu_k$. Then the next
approximation is set as $\fu_{k+1} = \ft \|\ft\|^{-1}$ where $\ft  = \fu_k + \fs $. Linear system (\ref{gras})
is shown to be better conditioned than the one with $A - \theta_k I$.
The same linear system represents the {\em correction equation} in the well-known Jacobi-Davidson method~\cite{slejpenvdvSIMAX96},
which in its turn can be viewed as an accelerated Inexact Newton method \cite{vdv}.
When $A$ is SPD and the leftmost eigenpairs are being sought, it has been proved
in \cite{notay} that the Preconditioned Conjugate Gradient (PCG) method can be employed in the solution
of the correction equation.

There are still a number of drawbacks that advises against using  pure Newton  method:
first, the choice of an initial vector. In the Jacobi-Davidson algorithm the Rayleigh-Ritz procedure implements
a sort of restart that in part solves this problem.
Second, even the projected Jacobian $(I - \fu_k \fu_k\t) (A - \theta_k I) (I - \fu_k \fu_k\t)$ in the correction equation 
is ill-conditioned, or at least more ill-conditioned than
the original matrix $A$, being its smallest eigenvalue smaller than $\lambda_{j+1}-\lambda_j$ when seeking the $j$-th eigenpair
(see Lemma \ref{lemmaNOTAY}).
In \cite{wusaad} the problem of finding a ``well-conditioned'' Jacobian matrix for the Newton method is considered by the 
authors, who describe some low-rank variants of the Jacobian of (\ref{nonlinear}) and perform a large set of numerical experiments
showing that the best choice is problem dependent.
Many authors have also tried to find a good  preconditioner for matrix $A - \theta_k I$ since it is the
key for efficient iterative solution of the correction equation. This remains an open problem (see for example \cite{knyazev98}).

Starting from the findings  in  \cite{bergamaschi-et-al-06} and \cite{bbm08sisc},
the main contribution of this paper is the development of a sequence of preconditioners $\{  P_k \}$
for the PCG solution of the Newton correction equation (\ref{gras}),
based on the BFGS update of a given initial preconditioner for the coefficient matrix $A$.
We will theoretically prove that the sequence of the preconditioned Jacobians will remain close to the identity matrix if the first 
preconditioned Jacobian is so. 
A similar approach has been used in \cite{xue} where a rank-two modification of a given preconditioner is
used to accelerated MINRES in the framework of the Inexact Rayleigh Quotient Iteration.

The BFGS formula as used in this paper is one more example of the strict connection between two
overlapping worlds: numerical linear algebra and optimization. Many papers have discussed 
this relationship. Among the others we refer to~\cite{dapuzzo} and the references therein.
Also the problem of finding efficient preconditioners for the linearized systems has become a crucial issue 
for the efficient implementation of the interior point method in constrained optimization, see e.g.
\cite{JerGon,velazco,bergonzil02,bgvzCoap05}. Often, the coefficient matrices of the linear systems to be solved at each Newton iteration
are  very close in structure and this motivates a number of works which study the possibility of
updating a given preconditioner to obtain with small computational effort  a new preconditioner
\cite{bbmSISC2011,bddmSINUM}.

To overcome the problem of the starting point, 
we also propose to start the Newton process after a small number of iterations of a Conjugate Gradient procedure
for the minimization of the Rayleigh Quotient  (DACG, \cite{bgp97nlaa}) to yield a good initial vector.

The combined DACG-Newton algorithm is used in the approximation of $p \in [10, 20]$ eigenpairs of a number
of matrices arising from various realistic applications of size up to $10^6$ and number of nonzeros up to $4 \times 10^7$.
Numerical results show that, in the solution of the correction equation, 
the PCG method preconditioned by BFGS displays much faster convergence than the same 
method when the preconditioner is kept fixed during the Newton process, in every test case.
Moreover, the proposed approach is shown to be competitive with the Jacobi-Davidson method.

The remaining of the paper proceeds as follows: in Section 2 we introduce the preconditioner; Section 3 is devoted
to the proof of the main theorem which states the closeness
of the preconditioned matrix to the identity matrix. In Section 4 we describe implementation details while Section
5 reports some numerical results of the proposed method for the eigensolution of the test matrices.
Section 6 reports comparisons against the Jacobi-Davidson method and Section 7 draws the conclusions.
%

\section{BFGS update of an initial preconditioner}
Following the idea described in \cite{bbm08sisc}, we propose a sequence of preconditioners for the Newton systems
using the BFGS rank-two update.
To precondition the initial Newton system 
$ J_0 \fs_0 = -\fr$,  where
\[ J_0 = (I - \fu_0\fu_0\t) (A - \theta_0 I)  (I - \fu_0\fu_0\t), \quad \fr =  -(A\fu_0 - \theta_0 \fu_0), 
\qquad \theta_0 = \fu_0\t A \fu_0\]
we chose to use a projected  incomplete Cholesky preconditioner with partial fill-in~\cite{saa91}: 

\noindent
$P_0 = (I - \fu_0\fu_0\t) \widehat{P}_0 (I - \fu_0\fu_0\t)$ \quad with \quad
$\widehat{P}_0 = \left(L L\t\right)^{-1}$ being $L =   IC({\lf}, \tau_{IC}, A)$ an incomplete triangular Cholesky factor of $A$,
with parameters $\lf$, maximum fill-in of a row in $L$, and $\tau_{IC}$ the threshold for dropping small elements in the factorization.
Then a sequence  of projected preconditioners for the subsequent linear systems $J_k \fs_k = -\fr_k$ may be  defined by using the BFGS formula as: 
	\begin{eqnarray} 
		P_{k+1} &=& (I - \fu_{k+1}\fu_{k+1}\t) \widehat{P}_{k+1}(I - \fu_{k+1}\fu_{k+1}\t), \qquad \text{where}
		    \nonumber \\
		     \wP &=& 
		\frac{\fs \fs\t}{\fs\t \fy}   +
		\left(I - \frac{\fs \fy\t}{\fs\t \fy}\right) \widehat{P}_k
		\left(I - \frac{\fy \fs\t}{\fs\t \fy}\right) 
	\end{eqnarray}
	and $\fs \equiv \fs_k$ is the solution of the $k$-th Newton system whereas $\fy \equiv \fy_k = \fr_{k+1} - \fr_k$.

	We propose here a simplification of the preconditioner update formula based on the well-known cubic convergence of the Newton method which implies that
	$ \|\fe_{k+1}\| \ll \|\fe_k\| $,
	in  a suitable neighborhood of the solution (i.e. for a suitable $\delta$ s.t. $\|\fe_k \| < \delta$).
	As a consequence also the residual norm satisfies $\|\fr_{k+1}\|  \ll \|\fr_k\| $. We can then approximate
	$\fy_k$ with $-\fr_k$ and write the preconditioner at level $k+1$ as (with $\fr \equiv \fr_k$):
	\begin{equation}
		\wP = 
		-\frac{\fs \fs\t}{\fs\t \fr}   +
		\left(I - \frac{\fs \fr\t}{\fs\t \fr}\right) \widehat{P}_k
		\left(I - \frac{\fr \fs\t}{\fs\t \fr}\right)  
	\label{lastprec}
	\end{equation}
%
	Theorem \ref{spd} of next Section will prove that the
	preconditioner defined in (\ref{lastprec}) is SPD if $\ww_k$ is so. 

\section{Theoretical analysis of the  preconditioner}
\subsection{Finding the smallest eigenpair}
The idea of the BFGS preconditioner is taken from the general analysis  in \cite{bbm08sisc,bbmmp11} where a sequence of preconditioners
 is devised in order to precondition the sequence of Newton systems for a general nonlinear problem.
One of the ``Standard Assumptions'' made in these papers was the nonsingularity of the Jacobian in the solution
of the nonlinear system.  
Here the situation is different, the Jacobian in the correction equation
$J(\fu) = (I - \fu\fu\t) (A - q(\fu)I) (I - \fu \fu\t)$
is singular whatever $\fu$, in particular it is singular when $\fu $ is equal to the exact eigenvector.
The theoretical analysis of the goodness of the preconditioner will be therefore completely different,
though obtaining similar results, than that proposed in~\cite{bbm08sisc,bbmmp11}.

We start by recalling some known results about convergence of the Newton  method for eigenproblems.
At every step of our Newton  method we approximately solve
\[ (I - \fu_k \fu_k\t) (A - \theta_k I) (I - \fu_k \fu_k\t) \fs  = - (A \fu_k - \theta_k \fu_k)  \]
where $\theta_k = \fu_k\t A \fu_k $, in the space orthogonal to $\fu_k$. Then we set
\begin{equation}
\label{beta}
 \fu_{k+1} = \frac{\fu_k + \fs}{\|\fu_k + \fs\|} =  \frac{\fu_k + \fs}{\sqrt{1 + \|\fs\|^2}} = \frac{\fu_k + \fs}{\beta}, 
\end{equation}
in view of $\fu_k\t \fs = 0$ and $\|\fu_k\| = 1$, 
where we have defined $\beta = \sqrt{1 + \|\fs\|^2}$.

The above mentioned Newton iteration is shown to converge cubically if the correction equation is solved exactly.
Since this is not the case when it is iteratively solved, we simply assume convergence, namely that for a suitable
$\delta > 0$ such that $\|\fe_0\| < \delta$ there is a constant $r < 1$ such that
\begin{equation}
\label{ekp1}
 \|\fe_{k+1}\| < r \|\fe_k\|, \qquad k = 0, \ldots
\end{equation}

\noindent
{\bf Notation}.  \\[-.3em]

\noindent
In the sequel we will indicate as $\fv_1$ the exact eigenvector corresponding to the smallest exact eigenvalue $\lambda_1$.
The error vector at step $k$ is denoted by $\fe_k = \fu_k - \fv_1$, while the error in the eigenvalue approximation is
$\varepsilon_k = \theta_k - \lambda_1 ( >0)$. It is easily proved that there is a constant $M$ independent of $k$ such that
\begin{equation}
	\label{varep}
 \varepsilon_k \le M \|\fe_k\|^2. 
\end{equation}
With $\rho(A)$ we mean the largest modulus eigenvalue of $A$ while $\lambda(A)$ refers to a generic eigenvalue of matrix $A$.
As the matrix norm of a symmetric matrix $A$ we will use the Euclidean norm 
$\|A\| \equiv \|A\|_2 = \rho(A) = \displaystyle \sup_{\fx \in \R^n, \fx \ne 0} \dfrac{\fx\t A \fx}{\fx\t \fx} $. 
\begin{Remark}
	\label{rem1}
At first  sight the Jacobian matrix in the correction equation is singular, but this does not matter since the PCG
algorithm is run within the subspace of vectors orthogonal to $\fu_k$ (in fact also $\fr\t \fu_k = 0)$. Thus, notion
of positive definiteness, eigenvalue distribution, condition number, norms, etc, apply as usual  but with respect
to matrices restricted to this subspace.
\end{Remark}

\noindent
The following Lemma will bound the extremal eigenvalues of $J_k$ in the subspace orthogonal to $\fu_k$. 
\begin{lemma}
	\label{lemmaNOTAY}
	There is a positive number $\delta$ such that if $\|\fe_k\| < \delta$ then
	\[ J_k = (I - \fu_k \fu_k\t) (A - \theta_k I) (I - \fu_k \fu_k\t) \]
	is SPD in the subspace orthogonal to $\fu_k$. Moreover the following bounds hold:
	\[ \frac{\lambda_2 - \lambda_1}{2} < \fz\t J_k \fz  < \lambda_n \]
	for every unit norm vector $\fz$ orthogonal to $\fu_k$.
\end{lemma}
\begin{proof}
	From Lemma 3.1 in \cite{notay} and the definition of $\varepsilon_k$, 
	we have that  $\displaystyle \min_{\fz \perp \fu, \|\fz\| = 1} \fz\t J_k \fz \ge \lambda_1 + \lambda_2 - 2 \theta_k =  \lambda_2 - \lambda_1 - 2 \varepsilon_k   $.
	Now using (\ref{varep}),
	\[\lambda_2 - \lambda_1 - 2 \varepsilon_k \ge \lambda_2 - \lambda_1 - 2 M \|\fe_k\|^2 
\ge \lambda_2 - \lambda_1 - 2 M \delta^2 > \frac{\lambda_2 - \lambda_1}{2} > 0\quad \left(\text{if  } \ \delta < 
\sqrt {\dfrac{\lambda_2 - \lambda_1}{4 M} }
\right),\]
showing that $J_k$ is SPD. The upper bound for the eigenvalues of $J_k$ is straightforward. 
\end{proof}

\noindent
The previous Lemma allows us to prove that the preconditioner defined in (\ref{lastprec}) is SPD, as stated in the following
Theorem.

\begin{theorem}
	\label{spd}
	If the correction equation is solved exactly, then any matrix $\widehat{P}_k$ defined 
	by (\ref{lastprec}) is SPD and hence $P_k$ is SPD in the subspace orthogonal to $\fu_k$.
\end{theorem}
\begin{proof}
	The proof is carried out by induction. $\widehat{P}_0$ is SPD being an incomplete Cholesky factorization of the SPD
	matrix $A$, then
	from $J_k \fs = -\fr$, we can write
	\begin{equation}
		\wP = 
		-\frac{\fs \fs\t}{\fs\t \fr}   +
		\left(I - \frac{\fs \fr\t}{\fs\t \fr}\right) \widehat{P}_k
		\left(I - \frac{\fr \fs\t}{\fs\t \fr}\right)  
		 =  \frac{\fs \fs\t}{\fs\t J_k \fs}   +
		 \left(I - \frac{\fs \fs\t J_k}{\fs\t J_k \fs}\right) \widehat{P}_k
		\left(I - \frac{J_k \fs \fs\t}{\fs\t J_k \fs}\right)
	\end{equation}
	We define $F =  \left(I - \dfrac{J_k \fs \fs\t}{\fs\t J_k \fs}\right)$ and note that,
	by previous Lemma, $\alpha = \fs\t J_k \fs > 0$ since $\fs\t\fu_k = 0$. 
Let now $\ww_k$ be SPD by induction hypothesis, then
	for every $\fz \ne 0$,
	\[ 
	\fz\t \wP \fz = \frac{(\fz\t \fs)^2}{\alpha} + \fz\t F\t \widehat{P}_k F \fz =
	\begin{cases}   \fz\t \ww_k \fz > 0 &  (\text{if} \ \fz\t \fs  =  0) \\
				      \dfrac{(\fz\t \fs)^2}{\alpha} + (F  \fz)\t 
	\widehat{P}_k (F \fz)  \ge \dfrac{(\fz\t \fs)^2}{\alpha} > 0, & (\text{if} \ \fz\t \fs  \ne  0) \\
\end{cases} \]
which proves that $\wP$ is SPD.  If we now take $\fz \perp \fu_{k+1}$, we have 
\[\fz\t P_{k+1} \fz = 
\fz\t(I - \fu_{k+1} \fu_{k+1}\t)  \wP (I - \fu_{k+1} \fu_{k+1}\t)\fz = \fz\t \wP \fz > 0, \] which  completes the proof.
\end{proof}

\noindent
Let us define the difference between the preconditioned Jacobian and the identity matrix as \[E_{k} = I - \Jk P_{k} \Jk.  \]
Since by definition we have $J_k  \fu_k = 0$ then $\fu_k$ is the eigenvector of $J_k$ corresponding to the zero eigenvalue.
Hence, since also $J_k^{1/2} \fu_k = 0$ the error $E_k$ can also be defined as
\[ E_{k} = I - \Jk P_{k} \Jk  = I - \Jk (I - \fu_k\fu_k\t) \ww_k (I - \fu_k\fu_k\t) \Jk = I - \Jk \ww_k \Jk. \]
The following technical lemma will bound the norm of $\ww_{k}$ in terms of that of $E_{k}$.
Being $\ww_k$ SPD we can define its norm in the space orthogonal to $\fu_k$ as
\[ \|\ww_k\| = \sup_{\fw \perp \fu_k, \fw \ne 0} \frac{\fw\t \ww_k \fw}{\fw\t \fw}.\]
\begin{lemma}
\label{lemma1}
	There is a positive number $\delta$ such that if $\|\fe_k\| < \delta$ then
\[ \|\ww_k \| \le \frac{2}{\lambda_2 - \lambda_1} \left(1 + \|E_{k}\|\right).\]
\end{lemma}
\begin{proof}
	Let $\fw \in \R^n$,  $\fw\t \fu_k = 0$, $\fw \ne 0$. Then since
$J_k$ (and thus $\Jk$) is SPD in the subspace orthogonal to $\fu_k$ the linear system
	$\Jk \fz = \fw$ has a unique solution $\fz$ such that $\fz \t \fu_k = 0$. Therefore
\begin{eqnarray*}
0 &<&  \frac{\fw\t\ww_k \fw }{\fw\t \fw  } = 
		\frac{\fz\t J^{1/2} \ww_k J^{1/2} \fz}{\fz\t J_k \fz} = 
		\frac{\fz\t (I - E_k) \fz}{\fz\t J_k \fz}  = \frac{\fz\t \fz}{\fz\t J_k \fz}\left(1 -  \frac{\fz\t E_k \fz}{\fz\t\fz}
		\right) \le \ \text{(by Lemma \ref{lemmaNOTAY}}) \\
 & \le &  \frac{2}{\lambda_2 - \lambda_1} \left(1 -  \frac{\fz\t E_k \fz}{\fz\t\fz} \right)
                                                         \le  \frac{2}{\lambda_2 - \lambda_1} \left(1 +  \left|\frac{\fz\t E_k \fz}{\fz\t\fz}\right| \right)
                                                         \le  \frac{2}{\lambda_2 - \lambda_1} \left(1 +  \|E_k\| \right).
\end{eqnarray*}
\end{proof}
\vspace{4mm}

\noindent
The next lemma will relate the norms of the difference $\fs$ and of the norm of the error vector $\fe_k$:
\begin{lemma}
	\label{lemma3}
	There exists a positive number $\delta$ s.t. if $\|\fe_k\| < \delta$ then
	\[  \|\fs\| \le 3 \|\fe_k\| \]
\end{lemma}
\begin{proof}
From (\ref{beta}) we have
	\begin{equation} 
		\fs = \beta \fu_{k+1} - \fu_k =
		\beta \fu_{k+1} - \beta \fv_1 + \beta \fv_1 -\fv_1 + \fv_1 - \fu_k =
		\beta \fe_{k+1} + (\beta-1) \fv_1 - \fe_k .
	\end{equation} 
	Hence, taking norms and using (\ref{ekp1}), 
	\begin{eqnarray}
		\|\fs\| &\le & \sqrt{1+\|\fs\|^2} \|\fe_{k+1} \|  + \sqrt{1 + \|\fs\|^2}-1 + \|\fe_k\|  \nonumber \\
			& \le & (1+\|\fs\|) \|\fe_{k+1} \|  + \frac{\|\fs\|^2}{2} + \|\fe_k\|  \nonumber \\
			& \le & (2+\|\fs\|) \|\fe_{k} \|  + \frac{\|\fs\|^2}{2}  .
		\label{eq15}
	\end{eqnarray}
The last 2nd order inequality, when solved for $\|\fs\|$ gives
\begin{equation}
\label{fs}
 \|\fs\| \le 1 - \|\fe_k\| - \sqrt{1 - 6 \|\fe_k\| + \|\fe_k\|^2}.
\end{equation}
Choosing $\delta < \dfrac{2}{15}$ implies $\sqrt{1 - 6 \|\fe_k\| + \|\fe_k\|^2}  \ge 1 - 4 \|\fe_k\|$, which, 
combined with (\ref{fs}) provides the desired result.
\end{proof}
\vspace{4mm}

\noindent
Now we need to prove that the distance between two consecutive Jacobians is bounded by a constant times the error vector:
\begin{lemma}
\label{lemma4}
	There exists a positive number $\delta$ s.t. 
	if $\|\fe_k\| < \delta$ then 
	\[ \| J_{k+1} - J_k \| \le  c_1 \|\fe_k\| \]
	for a suitable constant $c_1$.
\end{lemma}
\begin{proof}
	\begin{eqnarray}
		\label{Delta}
		\Delta_k &=& J_{k+1} - J_k = \nonumber \\  & = & 
		\left(I - \fu_{k+1} \fu_{k+1}\t\right) (A - \theta_{k+1} I) \left(I - \fu_{k+1} \fu_{k+1}\t\right)-
		\left(I - \fu_k \fu_k\t\right) (A - \theta_k I) \left(I - \fu_k \fu_k\t\right)  \nonumber \\
		& = & \left(\theta_k - \theta_{k+1} \right)I
		+ \fu_k \fu_k\t (A - \theta_k I) - \fu_{k+1} \fu_{k+1}\t (A - \theta_{k+1} I)  \nonumber \\
		& & \phantom{\left(\theta_k - \theta_{k+1} \right)I} 
		+ (A - \theta_k I) \fu_k \fu_k\t - (A - \theta_{k+1} I) \fu_{k+1} \fu_{k+1}\t  = \nonumber \\
                & = & \left(\theta_k - \theta_{k+1} \right)I + G + G\t
\label{eq22}
	\end{eqnarray}
where we set $G =   \fu_k \fu_k\t (A - \theta_k I) - \fu_{k+1} \fu_{k+1}\t (A - \theta_{k+1} I) $.
The first term in (\ref{eq22}) can be bounded, using (\ref{varep}),  as 
\begin{equation}
	\label{theta}
\|  \left(\theta_k - \theta_{k+1} \right)I \| = \theta_k - \theta_{k+1} = \varepsilon_{k} - \varepsilon_{k+1}  \le
\varepsilon_k \le M \|\fe_k\|^2 \le  \|\fe_k\| \quad \left(\text{if } \ \delta < \frac{1}{M}\right) .
\end{equation}
To bound $\|G\|$ recall that $\fu_{k+1} = \dfrac{\fu_k + \fs}{\beta}$ therefore
\[\fu_{k+1}\fu_{k+1}\t = \dfrac{\fu_k + \fs}{\beta} \dfrac{\fu_k\t + \fs\t}{\beta} = \frac{1}{1 + \|\fs\|^2}\left(
\fu_k \fu_k\t + \fu_k \fs\t + \fs \fu_k\t + \fs\fs\t\right)  = \frac{1}{\beta^2} \fu_k \fu_k\t + H\]
with $\|H\| \le \|\fs\| \left(2 + \dfrac{\|\fs\|}{1 + \|\fs\|^2} \right)  \le \dfrac{5}{2} \|\fs\| \le \dfrac{15}{2} \fe_k$, 
by Lemma \ref{lemma3}, then
\begin{eqnarray}
	G & = & -\frac{\fu_k \fu_k\t}{1 + \|\fs\|^2}  (A - \theta_{k+1} I) - H (A -\theta_{k+1} I)+ \fu_k \fu_k\t  (A - \theta_{k} I)  \nonumber \\
   & = & -\frac{\fu_k \fu_k\t}{1 + \|\fs\|^2}(\theta_k -\theta_{k+1}) I  + \frac{\|\fs\|^2}{1+\|\fs\|^2} (\fu_k \fu_k\t) (A - \theta_{k} I) - H (A -\theta_{k+1} I),
\end{eqnarray}
whence, using (\ref{theta}) and again Lemma \ref{lemma3}, 
\begin{equation}
	\label{G}
\|G\| \le  \|\fe_k\| + 9 \|\fe_k\|^2 \lambda_n + \dfrac{15}{2} \|\fe_k\| \lambda_n \le \left(1+9 \delta \lambda_n + \dfrac{15}{2} \lambda_n\right) \|\fe_k\| = c_2 \|\fe_k\| .
\end{equation}
Combining (\ref{Delta}), (\ref{theta}) and (\ref{G}) we have: 
\[ \|\Delta_k \| \le \|\fe_k \| + 2 \|G\| \le (1+2 c_2) \|\fe_k\|. \]
Setting $c_1 = 1+2 c_2$ completes the proof.
\end{proof}
\vspace{4mm}

\noindent
Before stating Theorem \ref{theo} we need to prove as a last preliminary result that also the difference between the square root
of two consecutive Jacobians is bounded  in terms of the norm of the error vector:
\begin{lemma}
\label{lemma5}
	Let $S_k = \JJ - \sJ$. Then there 
	is a positive number $\delta$ s.t.
	        if $\|\fe_k\| < \delta$ then
	\[ \|S_k\| \le c_3 \sqrt{\|\fe_k\|}. \]
	for a suitable constant $c_3$.
\end{lemma}
\begin{proof}
By squaring the equation $S_k + \sJ = \sqrt{J_{k} + \Delta_k}$ we obtain 
\begin{equation}
\label{2ndS}
S_k^2 +  \sJ S_k +  S_k \sJ  - \Delta_k  = 0.
\end{equation}
Let now $\fx$ be a normalized eigenvector of the symmetric matrix $S$ such that $S \fx = \eta \fx$. Premultiplying by $\fx\t$  and postmultiplying by $\fx$ 
equation (\ref{2ndS}) yields
	\begin{equation}
\label{eig2}
\eta^2 + 2 (\fx\t \sJ \fx) \eta - \fx\t \Delta_k \fx = 0.
	\end{equation}
This quadratic equation has two solutions:
\begin{equation}
\eta _{12} = -\fx\t \sJ \fx \pm \sqrt{(\fx\t \sJ \fx)^2 +  \fx\t \Delta_k \fx}. 
\label{2ndord}
\end{equation}
The smallest solution is not an eigenvalue of $S_k$ since from the definition of $S_k$, an eigenvalue $\eta$ of
$S_k$ would satisfy
\[ \eta = \eta \fx\t \fx = \fx\t S_k \fx =\fx\t  \JJ \fx -\fx\t  \sJ \fx \ge -\fx\t  \sJ \fx. \]
Then, considering the largest solution of (\ref{2ndord}) 
\begin{eqnarray*}
	\eta &=& -\fx\t \sJ \fx + \sqrt{(\fx\t \sJ \fx)^2 +  \fx\t \Delta_k \fx}  \\
      &=&   \frac{\fx\t \Delta_k \fx}{\fx\t \sJ \fx + \sqrt{(\fx\t \sJ \fx)^2 +  \fx\t \Delta_k \fx}} \le \sqrt{ \left|\fx\t \Delta_k \fx\right|}.
\end{eqnarray*}
       Let $(\overline{\eta}, \overline{\fx})$ be the eigenpair corresponding to the  largest modulus eigenvalue of $S$. Then
       \[ \|S_k\| = |\overline{\eta}| \le  \sqrt{ \left|\overline{\fx}\t \Delta_k \overline{\fx}\right|} \le \sqrt{c_1\|\fe_k\|}  = c_3 \sqrt{\|\fe_k\|}.  \]
	\end{proof}
	\vspace{4mm}

	\noindent
	We are finally ready to prove the main results of this Section. 
	The following theorem will state the so called {\em bounded deterioration} \cite{MR1678201} of the 
preconditioner at step $k+1$ with respect to that of step $k$, namely that the distance of the preconditioned matrix from the identity matrix at step $k+1$ is less or equal 
	than that at step $k$
	plus a constant that may be small as desired, depending on the closeness of $\fu_0$ to the exact solution.
	
\begin{theorem}
	\label{theo}
	Let $\delta_0$ be such that $\|E_0\| < \delta_0$, there is a positive number $\delta$ s.t.
	        if $\|\fe_0\| < \delta$ then
	\[ \|E_{k+1}\| \le \|E_k\| + K \se \]
	for a suitable constant $K$.
\end{theorem}
\begin{proof}
The distance of the preconditioned Jacobian from the identity matrix 
can be written as follows,
where we have defined $N = S_k \wP \JJ + \JJ \wP S_k  +  S_k \wP S_k $:
\begin{eqnarray}
	E_{k+1} &=& I - \JJ \wP \JJ  =  \nonumber \\
	        &=& I - \left(\sJ + S_k\right) \wP \left(\sJ  + S_k \right)  =  \nonumber \\
		&=& I - \sJ \wP \sJ - N  =  \nonumber \\
		       &=&  I -  \sJ \frac{\fs \fs\t}{\fs\t J_k \fs} \sJ   -
	                \sJ \left(I - \frac{\fs \fs\t J_k}{\fs\t J_k \fs}\right) \wk
			\left(I - \frac{J_k \fs \fs\t}{\fs\t J_k \fs}\right)  \sJ  - N
	=  \nonumber \\
		       &=&  I -   \frac{\sJ \fs \fs\t\sJ }{\fs\t J_k \fs}   -
	\left(I - \frac{\sJ \fs \fs\t \sJ }{\fs\t J_k \fs}\right) \sJ  \wk  J_k^{1/2}
			\left(I - \frac{\sJ \fs \fs\t\sJ}{\fs\t J_k \fs}\right)  - N. \nonumber \\
\end{eqnarray}
Now set $\fw = \dfrac{\sJ \fs}{\|\sJ \fs\|}$ and $W = I - \fw \fw\t$; $W$ is an orthogonal projector since $\|\fw\| = 1$.  Then
\begin{eqnarray}
\label{Ekp1}
	E_{k+1} &= &W - W  \sJ  \wk  \sJ W - N  \nonumber \\
	 & = &  W + W  (E_k - I) W - N  = W E_k W - N
\end{eqnarray}
To bound $\|N\|$ we will use Lemma \ref{lemma1} and Lemma \ref{lemma5}:
\[\|N\|  \le \frac{2}{\lambda_2 - \lambda_1} \left(1 + \|E_{k+1}\|\right) \left(2 c_3 \sqrt{\|\fe_k\|} \sqrt{\lambda_n} + c_3^2 \sqrt{\delta} \se  \right) = 
c_4 \left(1 + \|E_{k+1}\|\right)  \sqrt{\|\fe_k\|}. \]
Now taking norms in (\ref{Ekp1}) yields
\[\| E_{k+1} \| \le \|E_k\| + c_4 \left(1 + \|E_{k+1}\|\right)  \sqrt{\|\fe_k\|},\]
which can be rewritten as
\begin{equation}
\| E_{k+1}\| \left(1- c_4 \se \right) \le \|E_k\| + c_4   \se .
\label{recur}
\end{equation}
From (\ref{recur}), we derive a bound for $\|E_k\|$.  If $\sqrt{\delta} < \dfrac{1}{2 c_4}$ then 
\begin{equation}
	\label{boundE}	
	\| E_k\| \le 2 \|E_{k-1}\| + 1  \le \ldots \le 2^{k} \|E_0\| + 2^k -1 \le 
2^{k} \left(\delta_0 +1 \right) = c_5.
\end{equation}

Again from (\ref{recur}) and using the bound (\ref{boundE}) we finally have
\begin{eqnarray*}
\| E_{k+1}\|  &\le& \frac{\|E_k\| + c_4   \se }{1 - c_4 \se } \nonumber \\
              &\le& \left(\|E_k\| + c_4   \se \right) \cdot (1 + 2 c_4 \se ) \nonumber \\
	      &\le & \|E_k\| + 2 c_4  \se  c_5 + c_4 (1 + 2 c_4 \delta) \se  = \|E_k\| + c_4 \left(2 c_5 + 1 + 2 c_4 \delta\right) \se .
\label{recur1}
\end{eqnarray*}
Setting $K = c_4 \left(2 c_5 + 1 + 2 c_4 \delta\right)$ completes the proof.
\end{proof}
\vspace{4mm}

\begin{Remark}
	It is more usual to evaluate the goodness of a preconditioner by bounding the extremal eigenvalues of the preconditioned
	matrix (if SPD) instead of using norms. However, it is worth observing that the initial preconditioner can be 
	selected so as to give $\rho\left(\sn  P_0 \sn\right) < 2$. In such case, in the most common situation  we would have
	\[ \|E_{k}\| = \max \left\{\left|\lambda(E_k) \right|\right\} = 1 - \min \left\{\lambda\left(\sJ  P_k \sJ\right)\right\}\] so that minimizing
	$\|E_k\|$ is the same as maximizing the smallest eigenvalue of the preconditioned matrix.
\end{Remark}

\subsection{Computing several eigenpairs}
When seeking an eigenvalue different from $\lambda_1$, say $\lambda_j$, the Jacobian matrix changes as
\[ J_k = (I - Q Q\t) (A - \theta_k I)  (I - Q Q\t) \]
where $Q  = \left [ \fv_1 \ \fv_2 \ldots \fv_{j} \ \fu_k\right ] 	$ is the matrix whose first $j$ columns
are the previously computed eigenvectors.
Analogously, also the preconditioner must be chosen orthogonal to $Q$  as
\begin{equation}
\label{precQ}
 P_{k+1} = (I - Q Q\t) \widehat{P}_{k+1}(I - Q Q\t). 
\end{equation}
The theoretical analysis developed in the previous Section applies with small technical variants also in this case since it is readily
proved that $\JJ P_{k+1} \JJ = \JJ \wP \JJ$.
The most significant changes regard the definition of $\varepsilon_k = \theta_k - \lambda_j, \ \fe_k = \fu_k - \fv_j$ and
the statement of Lemma \ref{lemmaNOTAY} (and the proof of Lemma \ref{lemma4} that uses its results), 
namely the bound for the smallest eigenvalue of $J_k$ which in the general case  becomes:
\[       \fz\t J_k \fz >   \frac{\lambda_{j+1} - \lambda_j}{2}  \]
        for every unit norm vector $\fz$ such that  $Q\t \fz = 0$.

	\section{Implementation}
\subsection{Choosing an initial eigenvector guess}
As mentioned in Section 1, another important issue  in the efficiency of the Newton approach 
for eigenvalue computation is represented by the appropriate choice
of the initial guess.
We propose here to perform some preliminary iterations of another eigenvalue solver, in order
to start the Newton iteration `sufficiently' close to the exact eigenvector.
We chose as the `preliminary' eigenvalue solver DACG \cite{bgp97nlaa,berpinsar98,BergamaschiPutti02}, 
which is based on the preconditioned conjugate
gradient (nonlinear) minimization  of the Rayleigh Quotient.
This method has proven very robust, and not particularly sensitive to the initial vector, in the computation of a few eigenpairs of
large SPD matrices.
%
%
%
%
%
%
\subsection{Implementation of the BFGS preconditioner update}
In this section we give the main lines of the implementation
of the product of our preconditioner times a vector,
which is needed when using a preconditioned Krylov method. 
At a certain nonlinear  iteration level, $k$, we need
to compute $\fc = \P_{k} \fg_l,$
where $\fg_l$ is the residual of the linear system at iteration $l$.
Let us suppose we compute an initial preconditioner $\P_0$. Then,
at the initial nonlinear iteration $k = 0$,
we simply have 
$ \fc = \P_0 \fz_l $. 
At step $k+1$ the preconditioner $\wP$ is defined recursively by (\ref{lastprec}) while $P_{k+1}$
using  (\ref{precQ}) can be written as
	\begin{eqnarray}
		P_{k+1} &=& (I - Q Q\t) \wP (I - Q Q\t) =  \nonumber \\
			       &=& (I - Q Q\t)\left\{\left(I - \frac{\fs \fr\t}{\fs\t \fr}\right) \ww_k
		\left(I - \frac{\fr \fs\t}{\fs\t \fr}\right)  
	-\frac{\fs \fs\t}{\fs\t \fr}\right\} (I - Q Q\t)    .
	\label{BFGS}
	\end{eqnarray}
To compute vector $\fc$ first we observe that $\fg_l$ is orthogonal to $Q$ 
so there is no need to apply matrix $I - QQ\t$ on the right of (\ref{BFGS}).
Application of preconditioner $\wP$ to the vector  $\fg_l$ can be performed
at the price of $2k$ dot products and $2k$ daxpys as described in Algorithm \ref{bfgspk}.
The scalar products $\alpha_k = \fs\t_k \fr_k, $ which appear at the denominator of $\wP$, can be computed
once and for all before starting the solution of the $(k+1)$-th linear system.
Last, the obtained vector $\fc$ must be orthogonalized against the columns of $Q$ by a classical Gram-Schimdt procedure.
\begin{algorithm}[h!]
\caption{Computation of $\fc = P_k \fg_l$ for the BFGS preconditioner.}
	\label{DACGalg}
	\begin{list}{}{\topsep=1.0em\itemsep=0.0em\parsep=0.5em}
		\item {\sc Input:} Vector $\fg_l$, scalar products $\alpha_s = \fs\t_s \fr_s, \ s = 0, \ldots, k-1$.
								   \item $\fw := \fg_l $   
								   \item {\sc for} $s := k-1$ \ {\sc to}\   $0$   
			  \begin{list}{\theenumi.}{\usecounter{enumi}%
					                 \leftmargin=3.0em\itemsep=0.1em\topsep=-0.5em}
								   \item $ a_s  :=  \fs_s\t \fw/\alpha_s$    
								   \item $ \fw  :=  \fw - a_s \fr_s$    
									   \end{list}
								   \item {\sc end for}   
								   \item $\fc :=  \ww_0 \fw$    
								   \item {\sc for} $s := 0$ \ {\sc to}\  $ k-1$   
			  \begin{list}{\theenumi.}{\usecounter{enumi}%
					                 \leftmargin=3.0em\itemsep=0.1em\topsep=-0.5em}
								   \item $ b  :=  \fr_s\t \fc /\alpha_s$   
								   \item $ \fc  :=  \fc - (a_s + b) \fs_s  $   
									   \end{list}
								   \item {\sc end for}   
								   \item $\fz := Q\t \fc$
								   \item $\fc := \fc - Q \fz$
\end{list}
\label{bfgspk}
\end{algorithm}

\subsection{PCG solution of the correction equation}
As a Krylov subspace solver for the correction equation we chose the Preconditioned Conjugate gradient (PCG) method
since the Jacobian $J_k$ has been shown to be SPD in the subspace orthogonal to $\fu_k$.
Regarding the implementation of PCG, we mainly refer to the work \cite{notay}, where the author shows that it is possible
to  solve the linear system in the subspace orthogonal to $\fu_k$ and hence the projection step needed in the application 
of $J_k$ can be skipped. Moreover, we adopted the exit strategy for the linear system solution described in the above paper,
which allows for stopping the PCG iteration, in addition to the classical exit test based on a tolerance on
the relative residual and on the maximum number of iterations, whenever the current solution 
$\fx_{l}$ satisfies 
\begin{equation}
	\label{residual}
 \|\fr_{k,l}\| = \left \|A\fx_l - \dfrac{\fx_l\t A \fx_l}{\fx_l\t \fx_l} \fx_l \right\| < \tau \left(\fx_l\t A \fx_l\right) 
\end{equation}
or when the decrease of $\|\fr_{k, l}\|$ is slower
than the decrease of $\|\fg_l\|$, because in this case further iterating does not improve the accuracy of the eigenvector.
Note that this dynamic exit strategy implicitly defines an Inexact Newton method since the correction equation is not
solved ``exactly'' i.e. up to machine precision. 

We have implemented the PCG method as described in Algorithm 5.1 of \cite{notay}
with the obvious difference in the application of the preconditioner which is described here in Algorithm \ref{bfgspk}.
\subsection{Implementation of the DACG-Newton method}
The BFGS preconditioner defined in Algorithm \ref{bfgspk}
suffers from two main drawbacks, namely
increasing costs of memory for storing $\fs$ and $\fr$, and the increasing cost 
of preconditioner application with the iteration index $k$. Note that
these drawbacks are common to many iterative schemes, such as for example
sparse (Limited Memory) Broyden implementations~\cite{MR1112525}, 
GMRES~\cite{saadschultz86} and Arnoldi method for eigenvalue problems~\cite{ARPACK}.
There are different ways to overcome these difficulties, all
based on variations of a restart procedure, that is, the iteration scheme  
is reset after a fixed number of iterations. 
If the number of nonlinear iterations is high (e.g. more than ten iterations),
the application of BFGS preconditioner may be too heavy to
be counterbalanced by a reduction in the iteration number.
To this aim we  define $k_{\max}$ the maximum number of rank two
corrections we allow. 
When the nonlinear iteration counter $k$ is larger than
$k_{\max}$, the vectors $\fs_i, \fr_i,\  i = k-k_{\max}$ are substituted
with the last computed $\fs_k, \fr_k$. Vectors $\{\fs_i, \fr_i\}$
are stored in a matrix $V$ with $n$ rows and $2 \times k_{\max}$ columns.

The implementation of our DACG-Newton method for computing the leftmost eigenpairs
of large SPD matrices is described in Algorithm \ref{DACG-Newton}.
\begin{algorithm}
\caption{DACG-Newton Algorithm.}
\label{DACG-Newton}
\begin{itemize}{}{\topsep=1.0em\itemsep=0.0em\parsep=0.5em}
	\item {\sc Input:} 
	\begin{enumerate} \item \text{Matrix} $A$; 
			\item 			\text{number of sought eigenpairs} $n_{eig}$; 
			\item \text{tolerance and maximum number of its for the outer iteration:} $\tau$, {\sc itmax}; 
			\item \text{tolerance for the initial eigenvector guess } $\tau_{DACG}$; 
			\item \text{tolerance and maximum number of its for the inner iteration:} $\tau_{PCG}$, {\sc itmax$_{PCG}$}; 
			\item \text{parameters for the IC preconditioner:}, {\sc lfil} and $\tau_{IC}$;
			\item \text{maximum allowed rank-two updates in the BFGS preconditioner}: $k_{\max}$.
		\end{enumerate}
	\item $\wQ := [\ ] $.
	\item Compute an incomplete Cholesky factorization of $A$: $\widehat{P}_0$ with parameters ${\lf}$ and $\tau_{IC}$.
	\item {\sc for} $j := 1$ \ {\sc to}\   $n_{eig}$   
			  \begin{list}{\theenumi.}{\usecounter{enumi}%
					                 \leftmargin=3.0em\itemsep=0.1em\topsep=-0.5em}
							 \item Choose $\fx_0$ such that $\wQ\t \fx_0 = 0$.
							 \item Compute $\fu_0$, an approximation to $\fv_j$ by the DACG procedure
								 with initial vector $\fx_0$, preconditioner $\widehat{P}_0$ 
								 and tolerance $\tau_{DACG}$.

	\item $k := 0$, $\theta_k := \fu_k\t A \fu_k$.
\item {\sc while}\  $\| A \fu_k - \theta_k \fu_k  \| >  \tau \theta_k  $
\ {\sc and} \ $k < \IN $ \ {\sc do}
			  \begin{list}{\theenumi.}{\usecounter{enumi}%
					                 \leftmargin=3.0em\itemsep=0.1em\topsep=-0.5em}
 \item $Q := [\wQ \ \fu_k].$
\item
Solve   $J_k \fs_k = -\fr_k $ for $\fs_k \perp Q$
by the PCG method with preconditioner $P_k$ and tolerance $\tau_{PCG}$.
\item
	$\fu_{k+1} := \dfrac{\fu_k + \fs_k}{\| \fu_k + \fs_k\|}$, \ $\theta_{k+1} = \fu_{k+1}\t A \fu_{k+1}$.
\item $k_1 = k$ {\sc mod} $k_{\max}$; $V(*,2k_1+1) := \fs_k, \ V(*,2k_1+2) := \fr_k,$ 
\item
$k := k + 1$
			   \end{list}
		   \item {\sc end while}
		 \item Assume $\fv_j = \fu_k$ and $\lambda_j = \theta_k$.  Set $\wQ : = [\wQ \ \fv_j]$
			   \end{list}
		   \item {\sc end for}
  \end{itemize}
\end{algorithm}

\noindent

The above described implementation is well suited to parallelization provided that an efficient matrix-vector
product routine is available. The bottleneck is represented by the high number of scalar products
which may worsen the parallel efficiency when a very large number of processor is employed. Preliminary
numerical results are encouraging as documented in \cite{bm13jam}.

\section{Numerical Results}
In this Section we provide numerical results which compare the performance of the DACG-Newton algorithm for various $k_{\max}$ values.
We tested the proposed algorithm in the computation of the 20 smallest eigenpairs
of a number of  small to very large matrices arising from  various realistic applications.

The list of the selected problems together with their size $n$, and  nonzero number $nz$ is reported in Table
\ref{list}, where (M)FE stands for (Mixed) Finite Elements.

\begin{table}[h!]
	\label{list}
		\caption{Main characteristics of the matrices used in the tests.}

		\vspace{2mm}
	\begin{center}
		\begin{tabular}{llrr}
			\hline
		Matrix & where it comes from & $n$ & $nz$ \\
			\hline
    {\sc trino} & 3D-FE elasticity problem  & 4560 & 604030 \\
    {\sc hyb2d} & 2D-MFE groundwater flow  & 28600 & 142204 \\
    {\sc monte-carlo} & 2D-MFE stochastic PDE   & 77120 & 384320 \\
    {\sc emilia-923} & 3D-FE elasticity problem &  923136  &  41\,005206 \\
    {\sc dblp} & network connected graph  &  928498  &  8\,628378 \\
			\hline
	\end{tabular}
	\end{center}
\end{table}

\noindent
In most of the runs, unless differently specified,  we selected the values of the parameters as reported
in Table \ref{default}.
\begin{table}[h!]
	\label{default}
	\begin{center}
	\caption{Default values of parameters.}
	\vspace{2mm}
\begin{tabular}{ll}

	\hline

	Number of eigenpairs to compute &	$n_{eig} = 20$ \\ 
		    Parameters for the outer iteration & $\tau = 10^{-8}$, \quad {\sc itmax} $ = 100$ \\
	   Tolerance for the initial eigenvector guess &  $\tau_{DACG} = 10^{-2}$ \\
			Parameters for the PCG  iteration & $\tau_{PCG} = 10^{-2}$, \quad {\sc itmax$_{PCG} = 20 $} \\
	 Parameters for the initial preconditioner & ${\lf} = 30$,  \quad $\tau_{IC} = 10^{-2}$.\\
	\hline
\end{tabular}
	\end{center}
\end{table}

We will also compute the fill-in $\sigma$ of the initial preconditioner defined as
\[ \sigma = \frac
{\text{nonzeros of $L$}} 
{\text{nonzeros of lower triangular  part of $A$}}
\]
The CPU times (in seconds) refer to running a Fortran 90 code on an IBM Power6 at 4.7 GHz and with up to 64 Gb of RAM.

\subsection{Condition number of the preconditioned matrix}
We consider first the small {\sc trino} problem, for which  we were able to compute
all the eigenvalues of $\widetilde{J}_k = J_k^{1/2} P_{k}J_k^{1/2} $.
In Table \ref{condtrino} we report for each eigenvalue level $j$, the smallest eigenvalue of $\widetilde{J}_0$
and of $\widetilde{J}_5$, together with the ratio between condition numbers 
$\dfrac{\kappa(\widetilde{J}_0)}{\kappa(\widetilde{J}_5)}$ where 
\[ \kappa(\widetilde{J}_k) = \frac{\max \{ \lambda (\widetilde{J}_k) \}}{\min  \{ \lambda (\widetilde{J}_k)  > 0\}}, \]
and the reciprocal of the relative separation between consecutive eigenvalues,  $\xi_j = \dfrac{\lambda_j}{\lambda_{j+1}-\lambda_j}$
which is indicative of the ill-conditioning
of $\widetilde{J}_0$ \cite{bgp97nlaa,bmp12jam}.

\begin{table}[h!]
	\label{condtrino}
	\caption{Smallest eigenvalue of the preconditioned Jacobians: $\widetilde{J}_0$ and $\widetilde{J}_5$, condition number ratio
	and reciprocal of the relative separation $\xi_j$, for $j=1, \ldots, 20$. Matrix {\sc trino}.}
\begin{center}
	\begin{tabular}{r|ccrr||r|ccrr}
		\hline
	$j$ & $\lambda_{\min}(\widetilde{J}_0)$
	    & $\lambda_{\min}(\widetilde{J}_5)$ 
	    & $\dfrac{\kappa(\widetilde{J}_0)}{\kappa(\widetilde{J}_5)}$ &$ \xi_j$& 
	$j$ & $\lambda_{\min}(\widetilde{J}_0)$
	    & $\lambda_{\min}(\widetilde{J}_5)$ 
	    & $\dfrac{\kappa(\widetilde{J}_0)}{\kappa(\widetilde{J}_5)}$ &$\xi_j$\\
		\hline
		       1&       .0211&       .0542&        2.57&        1.57
		 &
		      11&       .0039&       .0192&        4.89&       27.99
		\\
		       2&       .0093&       .0432&        4.64&        7.55
		 &
		      12&       .0058&       .0151&        2.59&       31.45
		\\
		       3&       .0154&       .0403&        2.62&        3.79
		 &
		      13&       .0031&       .0127&        4.15&       34.85
		\\
		       4&       .0051&       .0256&        5.01&       13.69
		 &
		      14&       .0045&       .0285&        6.29&       36.78
		\\
		       5&       .0168&       .0417&        2.48&        5.91
		 &
		      15&       .0050&       .0284&        5.68&       24.74
		\\
		       6&       .0028&       .0316&       11.13&       37.25
		 &
		      16&       .0085&       .0234&        2.75&       30.27
		\\
		       7&       .0024&       .0349&       14.64&       51.94
		 &
		      17&       .0027&       .0197&        7.33&       81.48
		\\
		       8&       .0199&       .0329&        1.65&        5.13
		 &
		      18&       .0017&       .0307&       18.49&       51.52
		\\
		       9&       .0075&       .0149&        1.99&       22.62
		 &
		      19&       .0137&       .0378&        2.76&       11.28
		\\
		      10&       .0065&       .0196&        2.99&       19.11
		 &
		      20&       .0074&       .0296&        3.99&       22.98
		\\

		\hline

\end{tabular}
\end{center}
\end{table}

It is found that there is a constant  grow of the smallest eigenvalue from $k=0$ to $k=5$. Moreover
the condition number of the preconditioned matrix reduces by a factor varying from 1.65 ($j=8$) to 18.49 ($j=18$),
the reduction of the condition number being more important when $\xi_j$ is large, i.e. when the initial Jacobian
is ill-conditioned. Referring now to Theorem  3.7 we may observe that, being in this test case
$\lambda_{\max}(\widetilde{J}_k) \approx 1.5$ for every eigenpair,  \[\|E_k\| = \|I - \widetilde{J}_k\| = 
\max \{\lambda_{\max}(\widetilde{J}_k)-1 , 1 - \lambda_{\min}(\widetilde{J}_k)\}  = 1- \lambda_{\min}(\widetilde{J}_k). \] 
From Table \ref{condtrino}  we then could also easily compute $\|E_k\|$,
for $k = 0,5$.  For instance for $j=18$ we find $\|E_0\| = .9983 $ and $\|E_5\| = .9693$  
showing that even a small
reduction of $\|E_k\|$ may lead to an important reduction in the condition number of $\widetilde{J}_k$.

In Figure \ref{condnum} we plot the  condition number of the preconditioned Jacobian
for two  selected eigenpairs ($j=7, 18$),
vs outer iteration index using $k_{\max} = 0$, i.e. using $P_0$ as the preconditioner for all the Newton systems,
or $k_{\max} = 10$, i.e. using the BFGS preconditioner with no restart.
From the figure we notice that the condition number of  $\widetilde{J}_k$  remains roughly constant through the nonlinear iterations with
$P_k = P_0$ while it decreases  significantly if the BFGS preconditioner is employed.  

%
\begin{figure}[h!]
\label{condnum}
	\caption{Condition number of the preconditioned matrices $J_k^{1/2} P_{k}J_k^{1/2} $ 
vs outer iteration number, 
		in solving the {\tt trino} problem
	for $k_{\max}=0, 10$ and two different eigenpair levels ($j = 7$ and $j = 18$).}
	\vspace{2mm}
\begin{center}
\includegraphics[width=.6\textwidth]{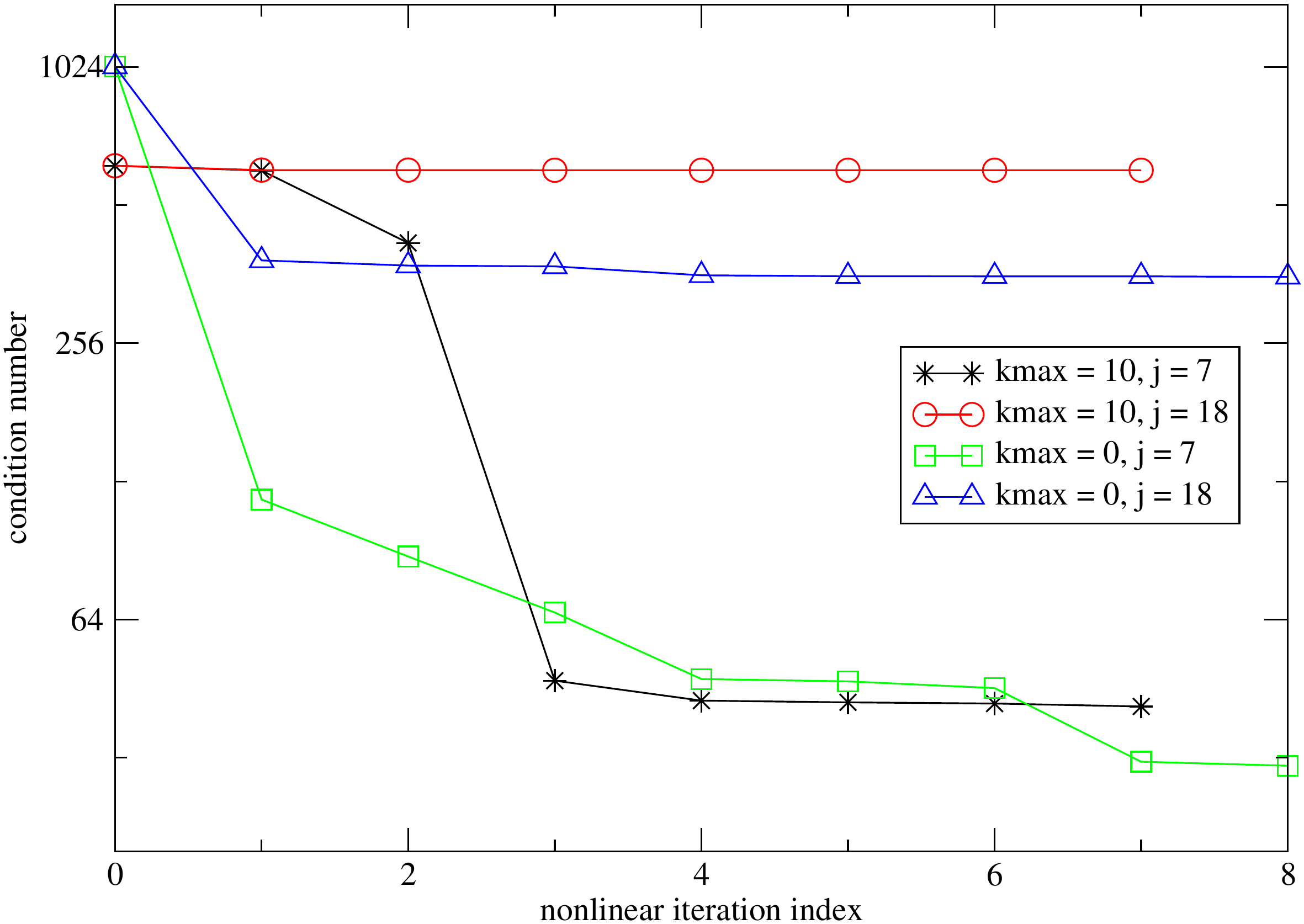}
\end{center}
\end{figure}

\subsection{Influence of parameter $k_{\max}$}
We now report  the results of our DACG-Newton method 
in the computation of the 20 leftmost eigenpairs of matrices {\sc hyb2d} and {\sc monte-carlo}.
The overall results are summarized in Tables \ref{tabcancun} and \ref{tabmonte} 
where  we include CPU times and number of matrix vector products (MVP)
for both the DACG initial phase and the Newton iteration. The overall outer Newton iterations (outer its) are also given.
From the tables  we notice that whatever the value of $k_{\max}$  there is an improvement in the total number of  MVP  and
CPU time compared with keeping  the initial preconditioner fixed throughout the nonlinear process.
The improvement is also irrespective of the maximum number of PCG iterations $\IL$.
Observe that increasing $k_{\max}$ also the outer iteration number decreases, meaning that
the proposed preconditioner, together with accelerating the correction equation solution, 
also speeds up nonlinear convergence to the desired eigenvector.
The CPU time values account for the fact that the overhead introduced by the application of $P_{k+1}$ for high
$k_{\max}$ values is not important. To this end it is worth emphasizing that in most cases $5$ outer iterations 
are enough to allow convergence of the Newton method and this also  explains why setting 
$k_{\max} \ge 5$  the results do not substantially change.
\begin{table}[h!]
\caption{Total number of iterations, number of MV products and CPU times for DACG-Newton algorithm
with various $k_{\max}$ values and two different values of $\IL$. Matrix {\sc hyb2d}. The initial preconditioner
fill-in is $\sigma = 2.19$.}
\vspace{2mm}
\label{tabcancun}
\begin{center}
\begin{tabular}{|c|r|cc|rrr||rr|}
	\hline
	&&	\multicolumn{2}{|c|}{DACG} &
	\multicolumn{3}{c||}{Newton} &
	\multicolumn{2}{c|}{TOT} \\
	\hline
	$\IL$&	$k_{\max}$ &  its & CPU & outer its & MVP & CPU & MVP & CPU \\
	\hline
  10 & 0    &944   & 3.03 &  261 &2730 & 7.85&3674 &     10.88  \\
     & 1    &944   & 2.98 &  160 &1647 & 4.93&2591 &      7.91  \\
     & 5    &944   & 2.97 &  110 &1134 & 3.50&2078 &      6.47  \\
 \hline
20 &  0    &944   & 2.99 &  118 &1879 & 5.15&2823    &  8.14  \\
    & 5    &944   & 2.99 &   87 &1297 & 3.77&2241    &   6.76  \\
\hline
 DACG &&  4033 &  11.67 & & & & 4033 & 11.67 \\
 \hline
\end{tabular}
\end{center}
\end{table}

\begin{table}[h!]
\caption{Total number of iterations, number of MV products and CPU times for DACG-Newton algorithm
with various $k_{\max}$ values and three different values of $\IL$. Matrix {\sc monte-carlo}. The initial preconditioner
fill-in is $\sigma = 2.10$.}
\vspace{2mm}
\label{tabmonte}
\begin{center}
\begin{tabular}{|c|r|cc|rrr||rr|}
	\hline
	&&	\multicolumn{2}{|c|}{DACG} &
	\multicolumn{3}{c||}{Newton} &
	\multicolumn{2}{c|}{TOT} \\
	\hline
	$\IL$&	$k_{\max}$ &  its & CPU & outer its & MVP & CPU & MVP & CPU \\
	\hline
 30 &  0  & 1921 &  15.98 &  148& 3969 & 30.10 & 4890 & 46.08 \\
    & 1   &1921  & 15.94  & 120 &2799    &22.06&4720  &38.00 \\
    & 3   &1921  & 15.81  & 102 &2264  &18.13 &4185   &33.94 \\
    & 10  & 1921 &  15.87 &  103& 2242 & 18.22 & 4163 & 34.09 \\
	\hline
 20 & 0   &1921  & 15.85  & 267 &5295  &39.88&7216    &55.73 \\
    & 1   &1921  & 15.84  & 155 &2998  &23.88&4919    &39.72 \\
    & 2   &1921  & 16.00  & 133 &2615    &21.23 &4536 &37.23 \\
    & 3   &1921  & 15.86  & 126 &2337  &19.16&4258    &35.02 \\
    & 5   &1921  & 15.78  & 114 &2121  &17.58 &4042   &33.36 \\
    & 10  & 1921 &  15.96 &  114 &2121 & 17.67& 4042  & 33.63 \\
	\hline
 10 & 0  &1921   &15.85  &1078&11783   &92.04&13704 &107.85 \\
    & 3  & 1921  & 15.79 &  259 &2740 &23.71& 4661   & 39.50 \\
    & 5  & 1921  & 15.93 &  230 &2434 &21.89 & 4355  & 37.82 \\
    & 10 &  1921 &  16.01&   193& 2042 & 18.84& 3963 &  34.85 \\
	\hline
	DACG & &7307 & 57.31 & & & &7307 & 57.31\\
	\hline
\end{tabular}
\end{center}
\end{table}
Considering for example the case $\IL=20$, matrix {\sc monte-carlo}, the number of MV products in the  Newton phase
reduces from an initial $5295$ to a final value of $2121$ obtained with $k_{\max} \ge 5$.
We finally observe that the DACG-Newton algorithm is always superior to ``pure'' DACG method (run up to  a relative
residual smaller than $\tau_{DACG} = 10^{-8}$) in terms of MV products and
CPU time as reported in the last row of both tables.

We also present two pictures (Figures \ref{monfig12} and \ref{monfig15}), where the relative residual norm $\|\fr_{k,l}\|$, computed by (\ref{residual}) is plotted vs
Newton cumulative linear iteration index for two selected eigenvalue levels. The problem is {\sc monte-carlo} with
$\IL = 20$. Again we let vary $k_{\max} \in [0,10]$.
The Figures confirm the  acceleration of convergence provided by the proposed preconditioner:
there is a factor $3 \div 4$ gain in the total number of iterations, when passing from $k_{\max}=0$ to $k_{\max} = 10$.
\begin{figure}[h!]
        \caption{Convergence profile of the relative residual norm vs cumulative inner iterations
        for eigenvalue \# 12, matrix {\tt monte-carlo.}}
	\vspace{2mm}
	\label{monfig12}
\begin{center}
\includegraphics[width=.7\textwidth]{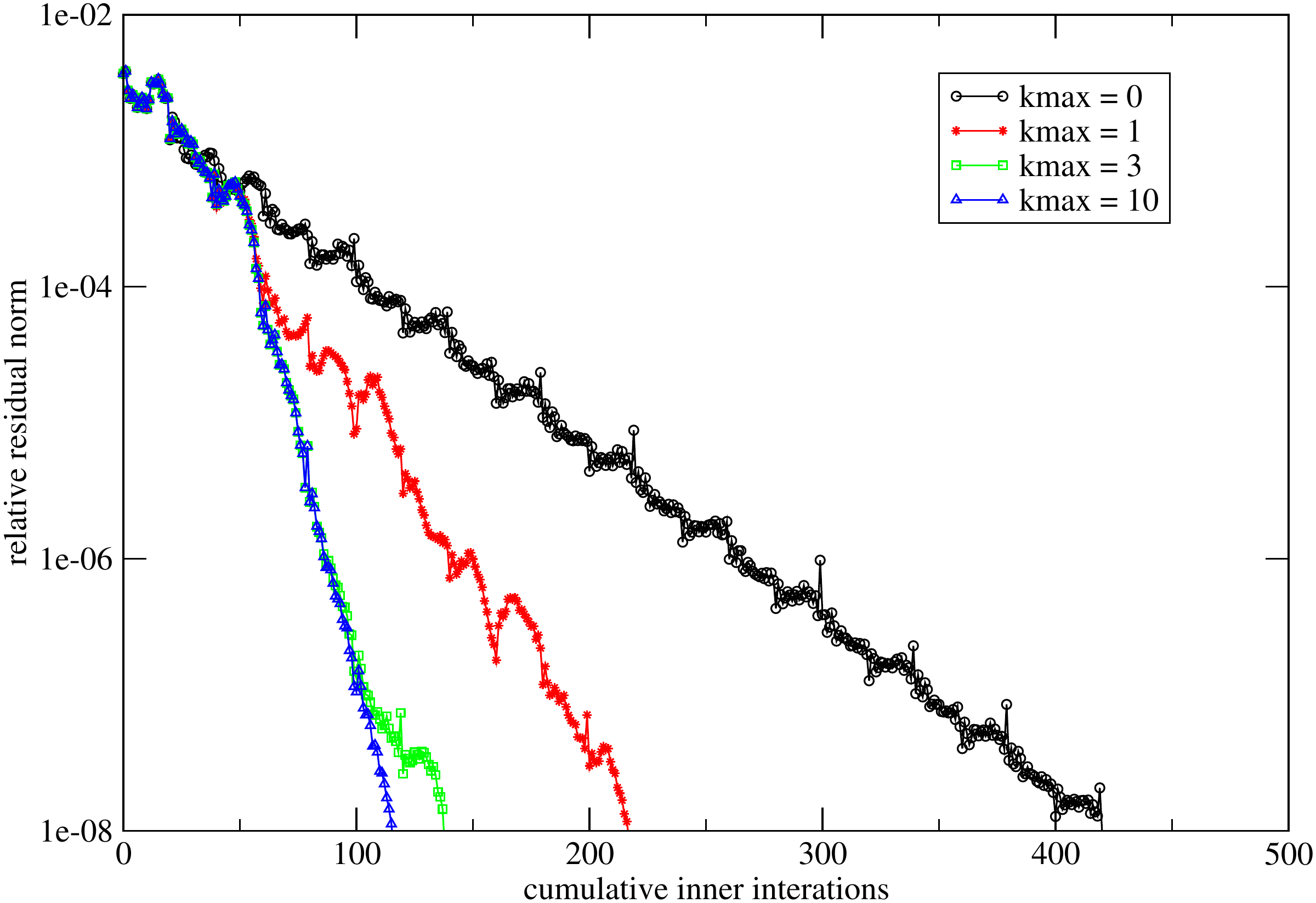}
\end{center}
\end{figure}

\begin{figure}[h!]
        \caption{Convergence profile of the relative residual norm vs cumulative inner iterations
        for eigenvalue \# 15, matrix {\tt monte-carlo.}}
	\vspace{2mm}
	\label{monfig15}
\begin{center}
\includegraphics[width=.7\textwidth]{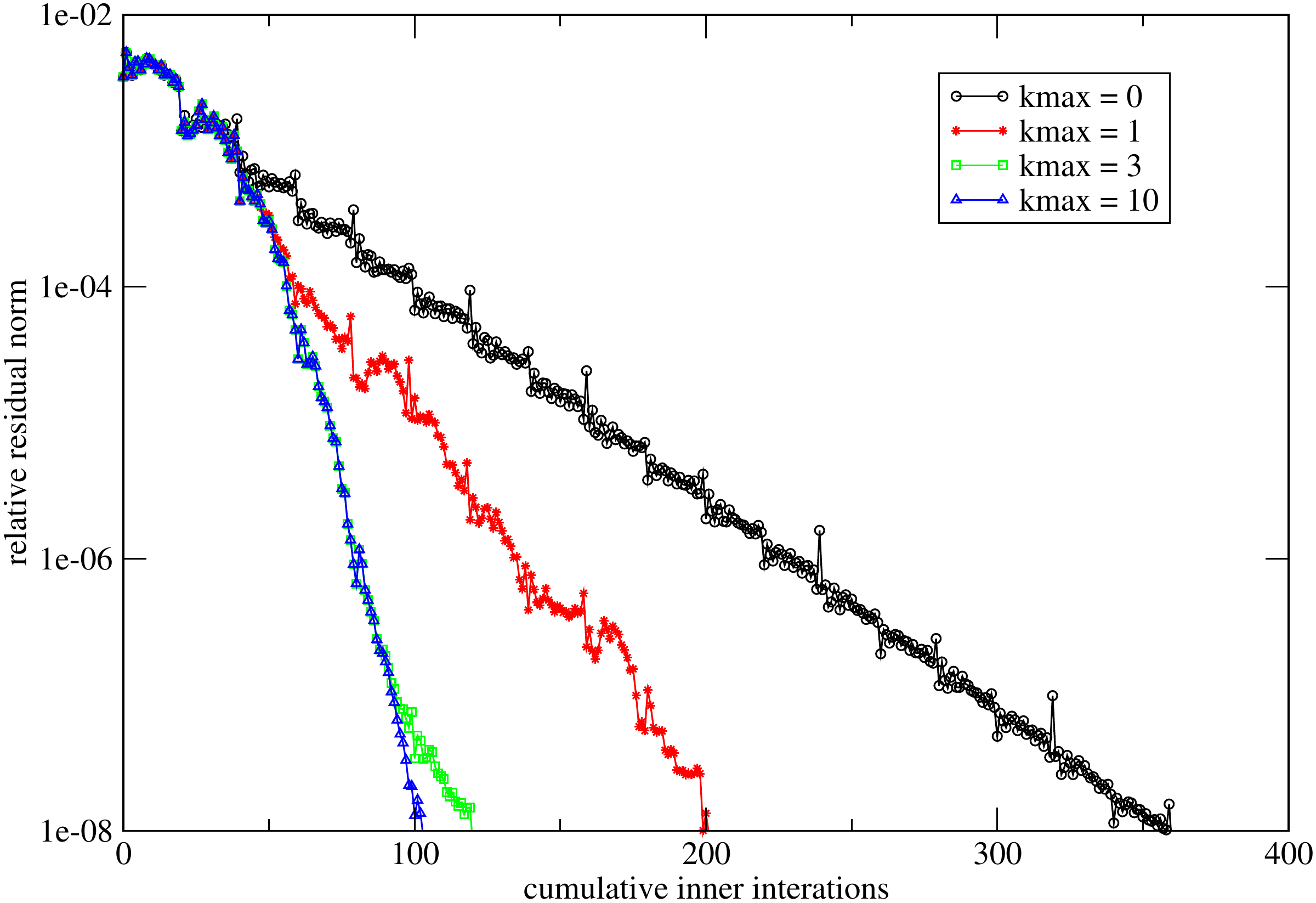}
\end{center}
\end{figure}

\subsection{Role of the initial preconditioner}
We report in this section results of the behavior of our preconditioner depending on the sparsity of $\ww_0$.
We selected two different sets of parameters for the IC preconditioner than those employed in Section 5.2,
namely ${\lf} = 10, \tau_{IC} = 0.1$ (test case \# 1, sparser preconditioner), and ${\lf}=50, \tau_{IC} = 10^{-4}$ 
(test case \# 2, preconditioner more filled-in).
For test case \# 1 we obtained $\sigma = 1.00$ and used $\IL = 30$  while for
    test case \# 2 we obtained $\sigma = 9.04$ and used $\IL = 10$.
\begin{table}[h!]
\caption{Total number of iterations, number of MV products and CPU times for DACG-Newton algorithm
with various $k_{\max}$. Matrix {\sc monte-carlo} with two different initial preconditioners.}
\vspace{2mm}
	\label{precondtab}
	\begin{center}
	\begin{tabular}{|rrrr|rrrr|}
\hline
		\multicolumn{4}{|l|} {test case \# 1 }                   &
	\multicolumn{4}{l|} {test case \# 2}                    \\
			\hline
			$k_{\max}$ &  outer its & MVP  & CPU 			 &
			$k_{\max}$ &  outer its & MVP  & CPU 			 \\
			\hline
		 0    &952 &28097& 174.62  &0    & 86  & 709 & 24.54 \\
		  1   & 340 &10032 & 64.09 &1    & 82  & 614  &20.79 \\
		   3  &  241 & 7008 & 47.94  &3  &   77  & 554  &19.49 \\
		    5 &   214 & 6097 & 42.64 &   5  &   77 &  554 & 20.69 \\
		    10&    193&  5419 & 38.69&    10 &    77 &  554&  20.69 \\
\hline
	\end{tabular}
	\end{center}
\end{table}
The results are summarized in Table \ref{precondtab} for the computation of the 20 
smallest eigenvalues of the {\sc monte-carlo} matrix.
We report only the results regarding the Newton phase, mentioning that in this case we chose $\tau_{DACG} = 0.1$.

The improvement in MV products/CPU time provided by the BFGS preconditioner is impressive with the sparser 
initial preconditioner, while in test case \# 2 the optimal number of iterations is reached with $k_{\max} = 3$.

\subsection{Eigensolution of the largest matrices}
We report the results obtained in evaluating $n_{eig} = 10$ eigenpairs of 
{\sc emilia-923} which  arises from the regional geomechanical model of a deep hydrocarbon
	reservoir.  This matrix is obtained discretizing the  structural problem
	 with tetrahedral Finite Elements.
	 Due to the complex geometry of the geological formation it was
	 not possible to obtain a computational grid characterized by
	 regularly shaped elements.
	 This matrix is publicly available in the University of Florida Sparse
	 Matrix Collection at {\tt http://www.cise.ufl.edu/ research/sparse/matrices}.
 To obtain an efficient initial preconditioner we selected ${\lf} = 50$ and $\tau_{IC} = 10^{-5}$  as the
 IC parameters which gave raise to a sparsity ratio $\sigma = 2.14$. DACG was run until a very high
 tolerance $\tau_{DACG} = 0.2$ 
 was reached. 
 
 In Table \ref{tabemilia} we report the MVP number and CPU time together with the CPU time taken 
 by the preconditioner construction, which is, in this case, a not negligible part of the overall
 computing time. The DACG-Newton method is shown to take great advantage from the BFGS preconditioner even for $k_{\max} =1$,
 displaying a halving of the MV products and CPU time with respect to reusing the same initial preconditioner.
 Note that in this case a very low DACG accuracy ($\tau_{DACG} = 0.2$) is sufficient to provide a good initial vector for the subsequent
 Newton phase.
\begin{table}[h!]
\caption{Total number of iterations, number of MV products and CPU times for DACG-Newton algorithm
with various $k_{\max}$ values for problem {\sc Emilia-923}.}
\vspace{2mm}
\label{tabemilia}
\begin{center}
\begin{tabular}{|c|r|r|cc|rrr||rr|}
        \hline
        & &   IC & \multicolumn{2}{|c|}{DACG} &
        \multicolumn{3}{c||}{Newton} &
        \multicolumn{2}{c|}{TOT} \\
        \hline
        $\IL$&  $k_{\max}$ & CPU &  its & CPU & outer its & MVP & CPU & MVP & CPU \\
        \hline
  20 & 0 & 219.1    &565   & 780.2 &  147 &2657& 3788.9&3232 &  4788.2    \\
     & 1 & 219.1    &565   & 780.2 &   75 &1273 &1781.0&1838  &  2780.3   \\
     & 5 & 219.1    &565   & 780.2 &   63 &922 & 1299.0&1487 &  2299.5   \\
 \hline
 DACG &&219.1&  7512 &10313.9 & & & & 7512 &10533.0\\
 \hline
\end{tabular}
\end{center}
\end{table}

Matrix {\sc dblp} represents the Laplacian of a graph  describing
collaboration network of computer scientists. Nodes are authors and edges are collaborations in published papers, 
the edge weight is the number of publications shared by the authors \cite{franceschet}. 
 It is well-known that the smallest eigenvalue is zero (with multiplicity one) corresponding to an eigenvector
 with all unity components. We are therefore interested in the 20 smallest strictly positive eigenvalues.
 This matrix is also characterized by a high clustering of the eigenvalues, which makes the problem  difficult to solve
 due to the consequent ill-conditioning of the Jacobian matrices.
 In our runs, we used default parameters listed in Table \ref{default} with the exception of $\tau_{DACG} = 0.1$. The 
preconditioner density was in this case $\sigma = 1.84$. 
 
 The results are reported in Table \ref{tabdblp}. DACG with $\tau_{DACG} = 10^{-8}$ could not converge to the desired eigenpairs, 
 namely for $j=12$ it reached the maximum number of iterations (5000) which produced stagnation 
 in the convergence to the subsequent eigenpairs. Also DACG-Newton with $k_{\max} = 0$ reached 
 the maximum number of outer iterations already at level $j=2$. Setting $k_{\max}=1$ was instead
 sufficient to lead DACG-Newton to convergence.
 Also in this case higher values of $k_{\max}$ provided faster convergence.

\begin{table}[h!]
\caption{Total number of iterations, number of MV products and CPU times for DACG-Newton algorithm
with various $k_{\max}$ values for problem {\sc dblp}. Symbol $\ddag$ stands for no convergence.}
\vspace{2mm}
\label{tabdblp}
\begin{center}
\begin{tabular}{|c|r|cc|rrr||rr|}
        \hline
        &&  \multicolumn{2}{|c|}{DACG} &
        \multicolumn{3}{c||}{Newton} &
        \multicolumn{2}{c|}{TOT} \\
        \hline
        $\IL$&  $k_{\max}$ &  its & CPU & outer its & MVP & CPU & MVP & CPU \\
        \hline
	20 & 0      & 620 & 376.7 &   \ddag   & \ddag  & \ddag           & -- & --  \\
    & 1    &620   & 376.7 &  172 &3015 &1696.8&3625  &  2073.5   \\
    & 3    &620   & 376.7 &  153 &2157 &1239.7&2777  &  1616.4   \\
     & 5    &620   & 376.7 &  191 &2126 &1189.0&2746  &  1565.7   \\
     & 10   &620   & 376.7 &  156 &1964 & 1147.7&2580 &  1524.4   \\
 \hline
 DACG &&  \ddag & \ddag & & & & -- &--\\
 \hline
\end{tabular}
\end{center}
\end{table}

\section{Comparison with Jacobi-Davidson}
The algorithm presented and analyzed in the previous sections is here compared with the well-known Jacobi-Davidson (JD)
method. For the details of this method we refer to the original paper \cite{slejpenvdvSIMAX96}, 
as well as to successive works \cite{statopulos,fokslevdv,notay}
which analyze both theoretically and experimentally a number of variants of this method.
In this paper, we followed the implementation suggested in the previously cited work \cite{notay}, i.e. 
we made use of the PCG method as the inner solver, with the same initial preconditioner
as that used in the DACG-Newton method. Also the exit  tests used in the two methods are identical 
for both the outer iteration and the inner PCG solver.
In the JD implementation two parameters are crucial for its efficiency namely $m_{\min}$ and $m_{\max}$, the smallest and the
largest dimension of the subspace where the Rayleigh Ritz projection takes place.
After some attempts, we found that $m_{\min} = 5$ and $m_{\max} = 10$ were on the average the optimal values  of
such parameters.
In all the examples and both solvers we set $\IL = 20$.

The results of the comparison are summarized in Table \ref{compar} where we also
specify the tolerance $\tau_{DACG}$ selected for each problem. It is found that the two methods behave very
similarly, being Jacobi-Davidson slightly more performing on all the problems with the exception of matrix {\sc dblp}.

\begin{table}[h!]
	\caption{Comparison between DACG-Newton and Jacobi-Davidson.}
	\vspace{2mm}
\label{compar}
	\begin{center}
		\begin{tabular}{|l|l|lrrr|rrr|}
\hline
			problem &$n_{eig}$ &  \multicolumn{4}{c|}{DACG-Newton} & 
	           \multicolumn{3}{|c|}{Jacobi-Davidson} \\
		   & & $\tau_{DACG}$ &  MVP & outer its & CPU & MVP &outer its & CPU \\
	\hline
	{\sc hyb2d} & 20 & $0.01   $ &  $2241     $ & 87 &  6.76  & 1543 & 150 &  5.95  \\
{\sc monte-carlo} &  20 & $0.01   $ & $4042      $ & 114 & 33.36  & 2833 & 178 & 28.53  \\
   {\sc dblp} & 20 &  $0.1    $ & $	   2580  $ &156 & 1524.40   & 2916 & 187 & 1698.00 \\
 {\sc emilia-923} & 10 &  $0.2    $ & $	  1487   $ &63 & 2299.50  & 1472 &  95 & 2242.40 \\
	\hline
\end{tabular}
	\end{center}
\end{table}

\begin{figure}
        \caption{Convergence profile of the relative residual norm vs cumulative inner iterations
	for DACG-Newton and JD methods. Matrix {\tt emilia-923}. Eigenvalues $j=1$ (top figure), $j=2$ (middle figure) and
        $j=8$ (bottom figure).}
	\vspace{2mm}
	\newpage
	        \label{em1}
		\begin{center}
			\includegraphics[width=.65\textwidth]{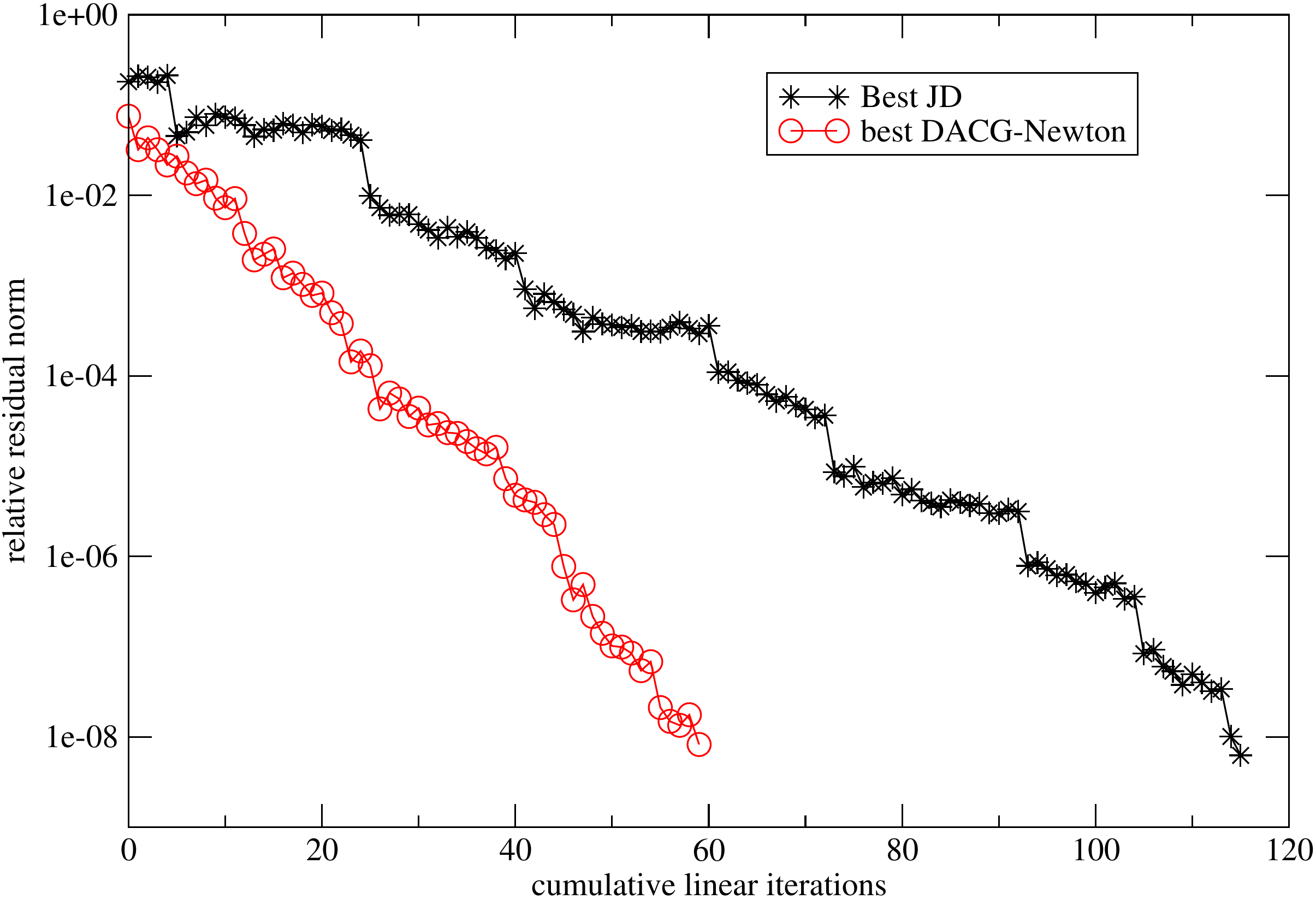}
		\end{center}
		\begin{center}
			\includegraphics[width=.65\textwidth]{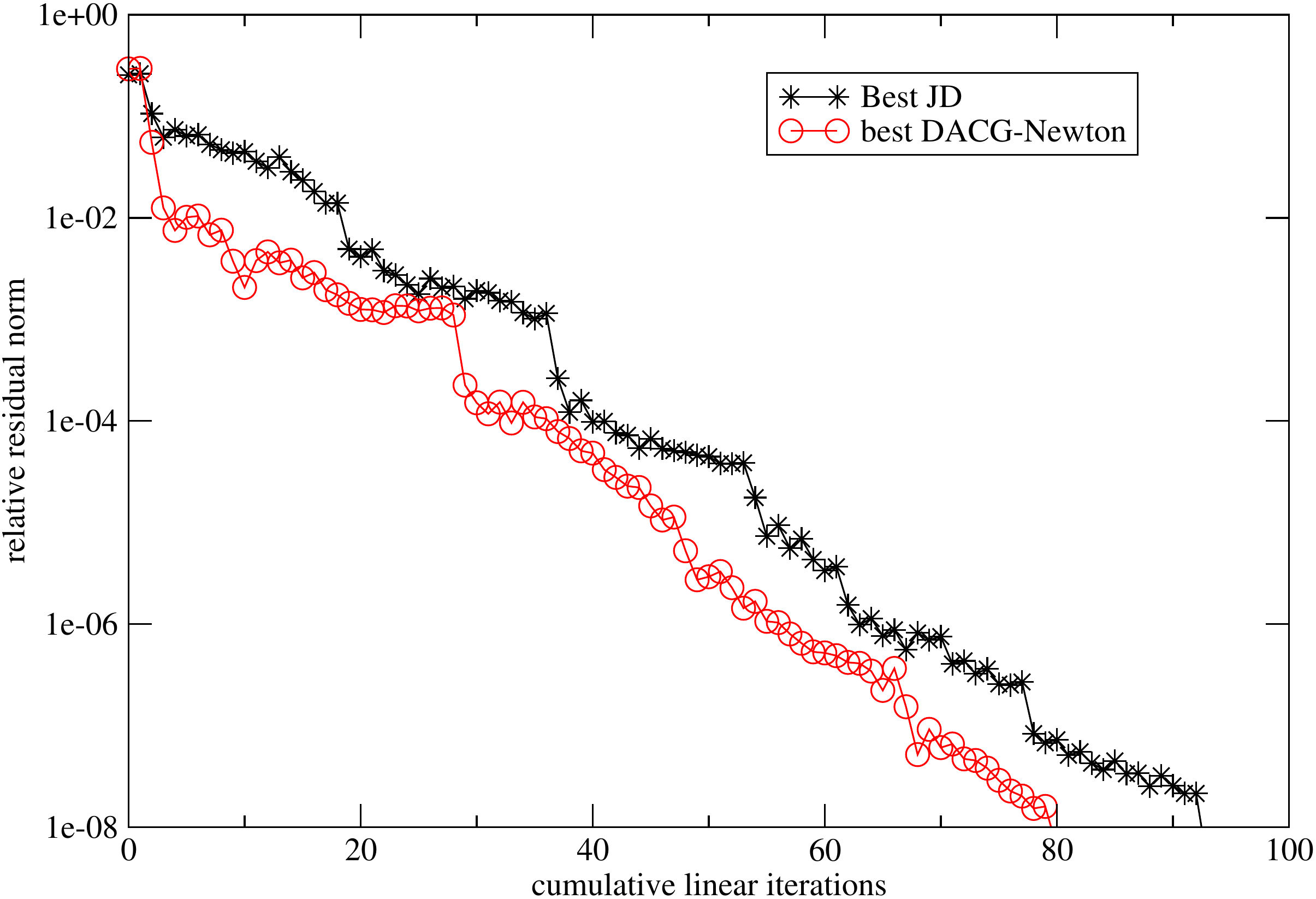}
		\end{center}
		\begin{center}
			\includegraphics[width=.65\textwidth]{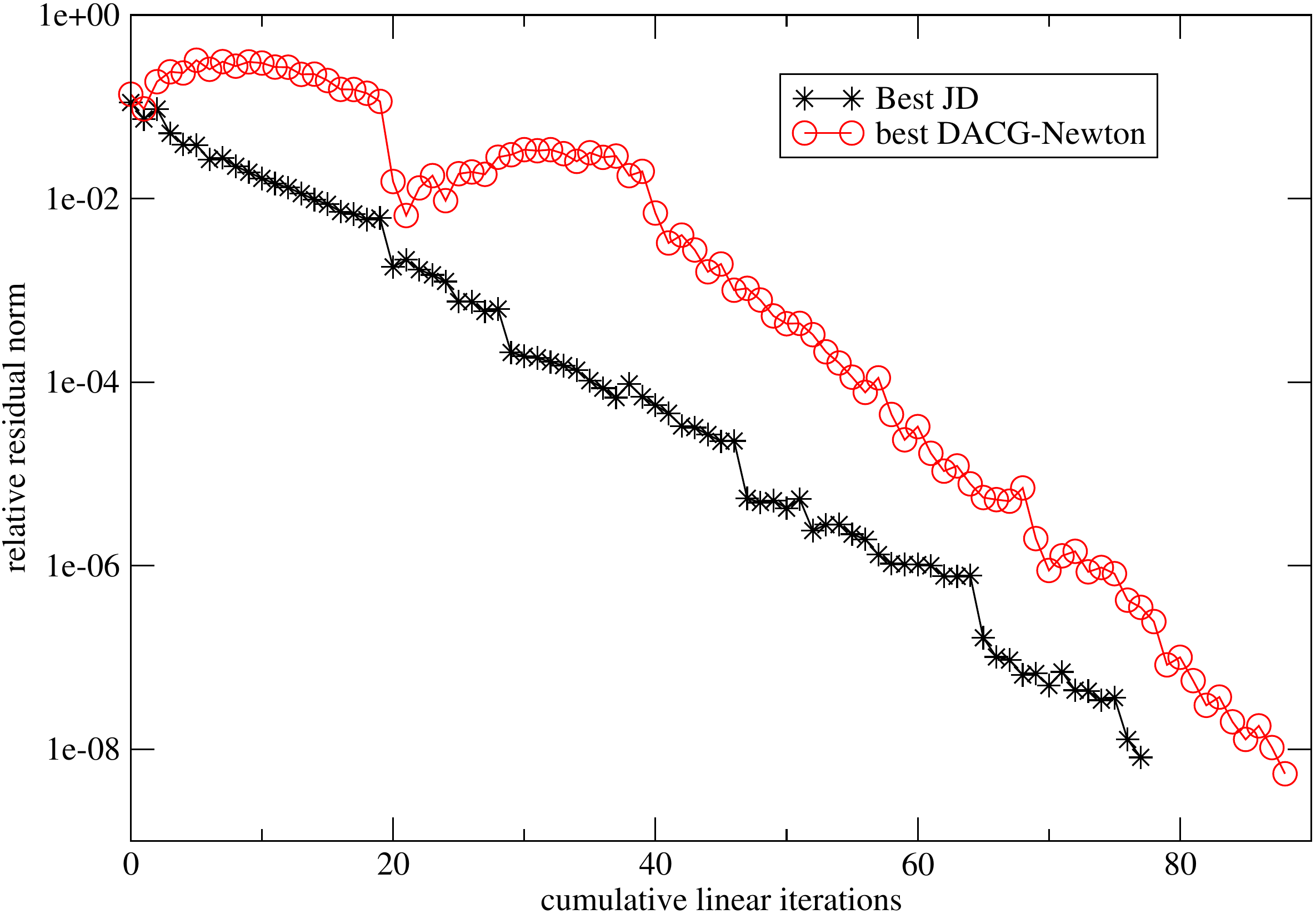}
		\end{center}
	\end{figure}

	Regarding problem {\sc emilia-923} we provide  in Figure \ref{em1} the  plot
of the relative residual norm vs cumulative linear iteration as per equation (\ref{residual}) for both DACG-Newton method
(Newton phase only) and JD,
corresponding to levels $j=1,2$ and $8$.
We can appreciate the very similar convergence profiles of these two methods. The DACG-Newton algorithm is faster for lower
$j$-values while the opposite holds for high values of $j$ where the Rayleigh-Ritz projection  seems to win against
the BFGS acceleration.

\section{Concluding remarks}
We have developed and theoretically analyzed a sequence of preconditioners aiming at accelerating the
PCG method in the solution of the correction equation. This equation is to be solved at each Newton iteration
to approximate a few eigenpairs of an SPD matrix.
Both theoretical analysis and experimental results onto a heterogeneous set of test matrices reveal that the BFGS 
sequence of preconditioners greatly improves the PCG efficiency as compared to using an initially evaluated fixed preconditioner.
The DACG-Newton method with the aforementioned preconditioner proves a robust and efficient algorithm
for the partial eigensolution of SPD matrices and makes ``pure'' Newton method competitive with Jacobi-Davidson
without making use of any Rayleigh-Ritz projection.
On the average the latter method proves a little bit more performing than the one proposed in this work,
and this is mainly due to the excessive DACG preprocessing time to devise a good initial vector for the subsequent Newton phase.
However, we wonder whether the preconditioning technique studied in this work may be seen as an alternative to Jacobi-Davidson or,
rather, a possible improvement of it.
At present we have neither theoretical nor experimental evidence in favor of this second option.
We therefore let as future work the attempt to insert our preconditioner in the framework of the Jacobi-Davidson method.
We will also compare the proposed algorithm with the recent implementations
of Inexact Arnoldi's (Lanczos') method \cite{FreiSpe09} where the inner linear system is solved with a variable accuracy
depending on the closeness to the wanted eigenvector.

The sequence of linear systems (\ref{gras}), the correction equation, has in common with the normal equations to be
solved at each interior point iteration the fact that the matrices involved get ill-conditioned as the iteration proceeds.
The BFGS sequence of preconditioners developed in this paper is expected to perform well also for preconditioning
the normal equations since:
\begin{enumerate}
	\item The bounded deterioration property, proved in Theorem \ref{theo}, is expected to mitigate the ill-conditioning of
		the linear systems toward the interior point solution.
	\item The approach described in the previous sections allows to perform only a (either complete or inexact)
	factorization of the initial Jacobian, thus saving on the cost of subsequent factorizations which
	is know to represent the main computational burden \cite{bergonzil02} of the whole interior point method for large
	and sparse constrained optimization problems.
\end{enumerate}

\end{document}